\documentclass{amsart}
\usepackage{amssymb, graphics, color, enumitem}
\usepackage[all]{xy}

 \usepackage{tikz}

 \usepackage[applemac]{inputenc}
\usepackage[cyr]{aeguill}

\usepackage[margin = 1in]{geometry}

\newtheorem{lemma}{Lemma}[section]

\newtheorem{theorem}[lemma]{Theorem}

\newtheorem{example}[lemma]{Example}

\newtheorem{remark}[lemma]{Remark}

\newtheorem*{Acknowledgement}{Acknowledgements}

\newtheorem{assumption}[lemma]{Assumption}

\newcommand\cf{cf\@. }

\newcommand\eg{e\@.g\@. }

\newcommand\pa{ \partial}

\newcommand\bbC{\mathbb C}

\newcommand\bbN{\mathbb N}

\newcommand\bbQ{\mathbb Q}
\newcommand\bbR{\mathbb R}
\newcommand\bbS{\mathbb S}

\newcommand\bbZ{\mathbb Z}

\usepackage{color}

\newcommand{\lrp}[1]{\left( {#1} \right)}

\newcommand\hH{\widehat{H}}

\newcommand\CI{\mathcal{C}^{\infty}}

\newcommand\cC{\mathcal{C}}

\newcommand\cA{\mathcal{A}}

\newcommand\db{\overline{\pa}}

\newcommand\phg{\operatorname{phg}}

\newcommand\hphi{\hat{\phi}}

\newcommand\AC{\operatorname{AC}}

\newcommand\SU{\operatorname{SU}}

\newcommand\QAC{\operatorname{QAC}}

\newcommand\nQb{\mathfrak{n}\operatorname{Qb}}

\newcommand\QALE{\operatorname{QALE}}

\newcommand\sing{\operatorname{sing}}

\newcommand\cO{\mathcal{O}}

\newcommand\Hilb{\operatorname{Hilb}}

\begin{document}
\title[New examples of affine Calabi-Yau $3$-folds]{New examples of affine Calabi-Yau $3$-folds with maximal volume growth}

\author{Shih-Kai Chiu}
\address{Department of Mathematics, University of California, Irvine}
\email{shihkaic@uci.edu}
\author{Ronan J.~Conlon}
\address{Department of Mathematical Sciences, The University of Texas at Dallas}
\email{ronan.conlon@utdallas.edu}

\author{Fr\'ed\'eric Rochon}
\address{D\'epartement de Math\'ematiques, Universit\'e du Qu\'ebec \`a Montr\'eal}
\email{rochon.frederic@uqam.ca}

\maketitle

\begin{abstract}
We construct examples of complete Calabi-Yau metrics on smoothings of $3$-dimensional Calabi-Yau cones that are not products of lower-dimensional Calabi-Yau cones and that have orbifold singularities away from the vertex.
\end{abstract}

\tableofcontents

\section{Introduction}

A complete Calabi-Yau manifold of complex dimension $m$ is of maximal volume growth when the volume of a ball of radius $r$ is comparable to $r^{2m}$ as $r$ tends to infinity.  For such a manifold, a tangent cone at infinity is necessarily of real dimension $2m$.  Moreover, by \cite{Colding-Minicozzi}, the tangent cone at infinity is unique when it has a smooth link, which is the case if and only if the curvature of the metric decays quadratically. By \cite{Sun-Zhang}, this implies that the metric is asymptotically conical ($\AC$ for short), meaning that it converges smoothly at infinity at a polynomial rate $\mathcal{O}(r^{-\epsilon})$ for some $\epsilon>0$ to a Calabi-Yau cone with a smooth link. By imposing symmetries, one can sometimes reduce the construction of a Calabi-Yau $\AC$-metric to finding a solution to an ordinary differential equation (ODE); see for instance \cite{Calabi, Stenzel}. HyperK\"ahler quotients also yield many examples, notably through the classification of asymptotically locally Euclidean (ALE) gravitational instantons by Kronheimer \cite{Kronheimer1989}.  Another approach consists in finding an approximate Calabi-Yau $\AC$-metric and solving a complex Monge-Amp\`ere equation to obtain an actual Calabi-Yau $\AC$-metric \cite{Joyce, vanC, vanC2011, Goto, CH2013,CH2015}.  Recently, a classification of Calabi-Yau $\AC$-metrics was obtained in \cite{CH3}: a Calabi-Yau $\AC$ manifold corresponds to a K\"ahler crepant resolution of a deformation of its tangent cone at infinity.

In general, however, a complete Calabi-Yau manifold of maximal volume growth can have a tangent cone at infinity with a singular link.  Similar to the $\AC$ case, various methods can be used to construct such metrics.  In \cite{BG1997}, Biquard and Gauduchon obtained explicit examples on the cotangent bundles of Hermitian symmetric spaces by imposing symmetries, reducing the construction to solving an ODE.  The hyperK\"ahler quotient construction also yields many examples, notably the Nakajima metric on the reduced Hilbert scheme $\Hilb_0^n(\bbC^2)$ of $n$ points in $\bbC^2$ by \cite{Nakajima,Carron2011}, and more generally generic quiver varieties and hyperK\"ahler toric varieties of finite topology by \cite{DR} and \cite{Rochon2025}.  The asymptotic geometry at infinity is very well understood in these cases.  For the Hilbert scheme $\Hilb_0^n(\bbC^2)$, the Nakajima metric is quasi-asymptotically locally Euclidean (QALE), while for quiver varieties and hyperK\"ahler toric varieties, the metrics are in the larger class of quasi-asymptotically conical (QAC) metrics.  

The class of $\QALE$-metrics was originally introduced by Joyce \cite{Joyce} in his construction of Calabi-Yau metrics on K\"ahler crepant resolutions of $\bbC^m/\Gamma$, where $\Gamma\subset \SU(m)$ is a finite subgroup that does not necessarily act freely on $\bbC^m\setminus\{0\}$.  The approach of Joyce consisted in gluing model metrics at infinity to produce a metric that is asymptotically Calabi-Yau;  the actual Calabi-Yau metric is then obtained by solving a complex Monge-Amp\`ere equation.  A key analytical step was to obtain good mapping properties for the Laplace-Beltrami operator of a $\QALE$-metric.  Motivated by the $\QALE$ case, Degeratu and Mazzeo \cite{DM2018} introduced the larger class of $\QAC$-metrics and derived the corresponding mapping properties of the Laplace-Beltrami operator. This made possible the construction in \cite{CDR2019} of Calabi-Yau $\QAC$-metrics on K\"ahler crepant resolutions of certain Calabi-Yau cones.  

Instead of crepant resolutions, one can consider smoothings of Calabi-Yau cones.  This idea was pursued independently in \cite{YangLi, CR2021,Szekelyhidi} to produce exotic examples of complete Calabi-Yau metrics of maximal volume growth on $\bbC^m$ for $m\ge 3$. In these examples, $\bbC^m$ is isomorphic to a suitable smoothing of the Calabi-Yau cone $C\times \bbC$, where $C \subset \bbC^{m}$ is a Calabi-Yau cone with a smooth link.  As pointed out in \cite{CR2021}, these metrics are not $\QAC$, but rather warped $\QAC$, since the asymptotic models near the singularities of the tangent cone are warped products rather than Cartesian products of $\AC$-metrics.  The construction in \cite{Szekelyhidi} was subsequently generalized in \cite{Firester} to allow the cone $C$ to be given by a complete intersection.  In all of these examples, the link of the tangent cone is singular, with singularities of depth at most 1.  By enlarging the definition of warped $\QAC$-metrics, examples with higher-depth singularities were produced in \cite{CR2024} by replacing $C$ by a product $C_1\times\cdots\times C_{\ell}$ of complete intersection Calabi-Yau cones with smooth links.  In a different direction, let us mention that Biquard and Delcroix \cite{Biquard-Delcroix}, and subsequently Nghiem \cite{Nghiem},  have generalized the construction of \cite{BG1997} by solving a complex Monge-Amp\`ere equation using the wonderful compactification of the underlying symmetric space.

In the examples of \cite{YangLi, CR2021,Szekelyhidi, Chiu, Firester, CR2024}, see also \cite{Yan}, the tangent cone at infinity is always a product of cones, one of whose factors is $\bbC$.  The goal of the present paper is to adapt the approach of \cite{CR2024} to situations in which the tangent cone at infinity has a singular link, but is not a product of lower-dimensional Calabi-Yau cones.  To describe our main result, Theorem~\ref{CY.14} below, let $P \in \bbC[z_1,z_2,z_3,z_4]$ be a homogeneous polynomial of weighted degree $d$ such that $P(z) = 0$ defines a Calabi-Yau cone $C_0 \subset \bbC^4$. Let us further assume that the singularities of $C_0$ are complex lines given by some of the coordinate axes and that, away from the apex, they correspond to quotient singularities with the Calabi-Yau cone metric smooth in the sense of orbifolds.  For such a cone, we consider smoothings of the form
$$
     P(z)=Q(z), \quad z\in \bbC^{4},
$$  
for a polynomial $Q$ of sufficiently high weighted degree $\ell$.  More precisely, if $d_s$ is the weighted degree of the polynomial defining the transverse Calabi-Yau orbifold, then we require $\ell\in [d-d_s, d)$. The point of this assumption is to write down a K\"ahler warped $\QAC$-metric with bounded holomorphic bisectional curvature, which is required for solving the complex Monge-Amp\`ere equation; see Remark~\ref{CY.16} below for more details. We then follow closely the strategy in \cite{CR2024}. The resulting Calabi-Yau metrics have curvature decaying at a rate $r^{-2\nu}$ for some specific $\nu\in[0,1)$ depending on the smoothing; see the statement of Theorem~\ref{CY.14} for the precise definition of $\nu$.  As explained in Remark~\ref{CY.4b}, our method only works in complex dimension $3$.   However, it can also work when the tangent cone at infinity is a complete intersection; see Theorem~\ref{ci.10} below for a specific example.

One way to generate new examples of Calabi-Yau metrics from this result is to consider, for  $k,l \in \bbN$, the variety
\begin{align*}
  C_{k,l} = \{ u^kv^l = wz\} \subset \bbC^4.
\end{align*}
These are $\bbQ$-Gorenstein toric varieties and admit Calabi-Yau cone metrics by \cite{CveticLuPagePope2005} and, independently, \cite{MartelliSparks2005}; see also \cite{Berman2023} for the general existence result.  Away from the apex of the cone, the metric is smooth in the sense of orbifolds thanks to Lemma~\ref{ap.1} below. Note that $C_{1,1}$ is the conifold, and if at least one of $k$ and $l$ is at least $2$, then $C_{k,l}$ has a one-dimensional singular set. Our main result applies to this class of Calabi-Yau cones and yields the following.

\begin{theorem}
  For both $k \ge 2$ with $l=1$, and $k=l \ge 2$, there exist smoothings of $C_{k,l}$ that admit complete Calabi-Yau metrics with maximal volume growth. These metrics are warped $\QAC$ in the sense of \cite{CR2024} and have $C_{k,l}$ as tangent cone at infinity.
\end{theorem}

 Note that $C_{2,1}$ is the suspended pinch point (SPP) singularity. The SPP singularity is perhaps the simplest non-product three-dimensional $\bbQ$-Gorenstein singularity with non-isolated singular set, and its singularities away from the origin are locally isomorphic to $\bbC \times \bbC^2/\bbZ_2$. Notably, the cones $C_{k,1}$ with $k\ge 2$ have irregular Reeb vector fields. 
 \begin{remark}
In complex dimension $m \ge 4$, examples of maximal volume growth Calabi-Yau metrics with irregular singular tangent cones at infinity have also been recently constructed by Nghiem \cite{Nghiem}.  
\end{remark}

The paper is organized as follows.  The main technical result, Theorem~\ref{CY.14}, is proved in \S~\ref{mr.0}, while the main examples are presented in \S~\ref{ex.0}.  In \S~\ref{ci.0}, we provide a specific example where the tangent cone at infinity is a complete intersection.

\begin{Acknowledgement}  
We are grateful to Vestislav Apostolov for pointing out to us Lemma~\ref{ap.1} and its proof, Tran-Trung Nghiem for useful discussions, and Matej Filip for bringing our attention to the cones $C_{k,l}$. The second author was supported by a Simons Travel Grant, and the third author was supported by a NSERC discovery grant and a FRQNT team research project grant. Part of this work was carried out when all three authors were attending the workshop \emph{Special Geometric Structures and Analysis} that took place at SLMath (formerly MSRI) in September 2024. 

\end{Acknowledgement}

\section{Main result} \label{mr.0}

Let $(C_0,g_{C_0})$ be a Calabi-Yau cone with $C_0$ corresponding to a codimension 1 hypersurface in $\bbC^{4}$ given by 
\begin{equation}
 C_0= \{ z\in \bbC^{4}\; | \; P(z)=0\}
\label{CY.1}\end{equation}
for some polynomial $P$.  Suppose that  the natural $\bbR^+$-action on $C_0$ is induced by an $\bbR^+$-action on $\bbC^{4}$ of the form
\begin{equation}
      \begin{array}{lccl}
      \bbR^+\ni t: & \bbC^{4} & \to & \bbC^{4} \\
             &z &\mapsto & t\cdot z =(t^{w_1}z_1, t^{w_2}z_2, t^{w_3}z_3,t^{w_{4}}z_{4})
      \end{array}
\label{CY.2}\end{equation}
for some multiweight $w=(w_1, w_2,w_3,w_{4})\in (\bbR^+)^{4}$.  We will not assume that the Calabi-Yau cone $(C_0,g_{C_0})$ is quasi-regular, so the weights $w_i$ are not necessarily rational.  Suppose also that the polynomial $P$ is homogeneous of (weighted) degree $d$ with respect to this $\bbR^+$-action, namely that
$$
   P(t\cdot z)= t^d P(z) \quad \forall \; t\in\bbR^+, \; z\in \bbC^{4}.
$$ 
On $C_0$, the K\"ahler form of the metric $g_{C_0}$ is $\omega_{C_0}= \frac{\sqrt{-1}}2\pa \db r^{2}_{C_0}$, where $r_{C_0}$ is the radial distance to the origin with respect to the Calabi-Yau metric $g_{C_0}$.  We will suppose that the singular locus of the Calabi-Yau cone $C_0$ takes the following form.
\begin{assumption}
The singular locus $C_{0,\sing}$ of $C_0$ is of the form
\begin{equation}
  C_{0,\sing}= \bigcup_{i\le s} L_i
\label{CY.3a}\end{equation}
for some $0<s\le 4$, where 
\begin{equation}
  L_i:= \{ z=(z_1,z_2,z_3,z_{4})\in \bbC^{4}\; | \; z_j=0 \; \mbox{for} \; j\ne i\}
\label{CY.3b}\end{equation}
is the $z_i$-axis.  
Moreover, we will suppose that $C_{0}\setminus \{0\}$ is a complex orbifold and that $g_{C_0}$ is smooth in the sense of orbifolds on $C_0\setminus \{0\}$.
\label{CY.3}\end{assumption}   
\begin{remark}
In principle, our method would work in complex dimension $m>3$ by requiring that the orbifolds singularities be of codimension $m-1$ and locally modelled on $\bbC\times (\bbC^{m-1}/\Gamma)$ for some finite subgroup $\Gamma_i\subset \SU(m-1)$ and with $\bbC^{m-1}/\Gamma$ admitting a description as an affine hypersurface.  However, by \cite{Gordeev1982,Kac-Watanabe1982}, the singular set of such a quotient $\bbC^{m-1}/\Gamma$ is at most of codimension $2$, forcing $m=3$ to have a non-trivial example.
\label{CY.4b}\end{remark}

For a polynomial $Q$ of weighted degree $\ell<d$, consider for $\epsilon\in \bbC$ the affine deformation
\begin{equation}
   C_\epsilon:= \{ z\in \bbC^{4}\; | \; P(z)=\epsilon Q(z) \}
\label{CY.4}\end{equation} 
of $C_0$ and suppose that $C_{\epsilon_0}$ is smooth for some $\epsilon_0\in \bbC\setminus \{0\}$.

Let $\overline{\bbC^{4}_w}$ be the weighted radial compactification of $\bbC^{4}$ with respect to the $\bbR^{+}$-action \eqref{CY.2} as described in \cite[\S~5]{CR2024}.  Let $\overline{C}_{\epsilon}$ denote the closure of $C_{\epsilon}$ in $\overline{\bbC^{4}_w}$.  Similarly, let $\overline{L}_i$ be the closure of $L_i$ in $\overline{\bbC^{4}_w}$.  Since $\ell<d$, notice that the boundary $\pa \overline{C}_{\epsilon}= \overline{C}_{\epsilon}\cap \pa \overline{\bbC^{4}_w}$ for $\epsilon\ne 0$ coincides with $\pa \overline{C}_0$.  In particular, even if $C_{\epsilon_0}$ is smooth, its closure $\overline{C}_{\epsilon_0}$ is not with singular locus corresponding to 
$$
    \pa\overline{C}_{0,\sing}= \bigcup_{i\le s} \pa\overline{L}_i.
$$  

To see this, we can consider the system of coordinates $(\xi_i, \omega_i)$
given by 
\begin{equation}
    \xi_i:= \frac{1}{|z_i|^{\frac{1}{w_i}}}\quad \mbox{and}  \quad \omega_i:=(\omega_{i,1},\omega_{i,2},\omega_{i,3},\omega_{i,4})\quad \mbox{with} \quad \omega_{i,j}:= \frac{z_j}{|z_i|^{\frac{w_j}{w_i}}}.
\label{CY.4a}\end{equation}
It is valid on $\overline{\bbC^{4}_w}$ near $\pa \overline{\bbC^{4}_w}$, but away from the hyperplane $z_i=0$.  As $i$ varies, this gives four systems of coordinates covering $\pa\overline{\bbC^4_w}$.  If 
$$
  Q(z)= \sum_{q} Q_q(z)
$$
is the homogeneous decomposition of $Q$ with $Q_q$ homogeneous of weighted degree $q$, then in the coordinates \eqref{CY.4a}, the equation defining $C_{\epsilon}$ takes the form
\begin{equation}
 P(\omega_i)= \epsilon \left( \xi_i^{d-\ell} Q_{\ell}(\omega_i) +\sum_{q<\ell} \xi_i^{d-q}Q_q(\omega_i) \right).
\label{CY.5}\end{equation}
At $\xi_i=0$, that is, on $\pa \overline{\bbC^{4}_w}$, this gives the equation
\begin{equation}
  P(\omega_i)=0
\label{CY.6}\end{equation}
with singularities along $\pa L_i$ if $i\le s$.

To resolve these singularities at infinity, we will consider suitable weighted blow-ups of $\pa\overline{L}_1,\ldots, \pa\overline{L}_s$.  The multiweights used in these weighted blow-ups will depend on the local description of the singularities.  To describe those,  consider the polynomial
$$
        P_{i,\omega_{i,i}}(\omega_{i,i_1},\omega_{i,i_2}, \omega_{i,i_3}):= P(\omega_{i,1},\omega_{i,2},\omega_{i,3},\omega_{i,4}),
$$ 
where $\{i_1,i_2,i_3\}=\{1,2,3,4\}\setminus\{i\}$.  For $\omega_{i,i}\in\pa\overline{L}_i\cong \bbS^1$ fixed, the equation 
\begin{equation}
  P_{i,\omega_{i,i}}(\omega_{i,i_1}, \omega_{i,i_2}, \omega_{i,i_3})=0
\label{CY.15}\end{equation}
is a singular affine variety $C_{i,\omega_{i,i}}$ describing the singularity of $C_0$ in the normal direction along the real half-line generated by $\omega_{i,i}$ in $L_i$.  By Assumption~\ref{CY.3}, the only singularity of $C_{i,\omega_{i,i}}$ is at the origin and it is an orbifold singularity, so that  
$$
     C_{i,\omega_{i,i}}\cong \bbC^{2}/\Gamma_{i,\omega_{i,i}}
$$   
for a finite group $\Gamma_{i,\omega_{i,i}}\subset \SU(2)$ acting freely outside of the origin.  Moreover, the metric induced by $g_{C_0}$ on $C_{i,\omega_{i,i}}$ corresponds to the Euclidean metric $\bbC^{2}/\Gamma_{i,\omega_{i,i}}$ under this identification.  In particular, as a Calabi-Yau cone, $C_{i,\omega_{i,i}}$ comes endowed with an $\bbR^+$-action.  
 \begin{assumption}
The $\bbR^+$-action on $C_{i,\omega_{i,i}}$ is induced by an $\bbR^+$-action on $\bbC^3$ of the form
$$
   t\cdot(\omega_{i,i_1},\omega_{i,i_2},\omega_{i,i_3})= (t^{v_{i,i_1}}\omega_{i,i_1}, t^{v_{i,i_2}}\omega_{i,i_2}, t^{v_{i,i_3}}\omega_{i,i_3})
$$
for some multiweight $v_i=(v_{i,i_1}, v_{i,i_2},v_{i,i_3})\in(\bbR^+)^3$.  With respect to this action, the polynomial $P_{i,\omega_{i,i}}$ is homogeneous of weighted degree $d_s>0$.  We will suppose that $d_s$ is independent of $\omega_{i,i}\in \pa\overline{L}_i$ and $i\in\{1,\ldots,s\}$.
\label{CY.13}\end{assumption}

For the weighted blow-up of $\pa\overline{L}_i$, this assumption gives us a multiweight  which assigns the weight $v_{i,j}$ to $\omega_{i,j}$ for $j\ne i$, while \eqref{CY.5} suggests that  the boundary defining function $\xi_i$ should be assigned the weight $\frac{d_s}{d-\ell}$.  Alternatively, denoting by $x_{\max}$ a global choice of boundary defining function for the manifold with boundary $\overline{\bbC^4_w}$, consider first the manifold with boundary $\widetilde{\bbC^{4}_w}$ which, as a topological space, is homeomorphic to $\overline{\bbC^{4}_w}$, but with smooth functions on $\widetilde{\bbC^{4}_w}$ corresponding to smooth functions on $\bbC^{4}$ having a smooth expansion in integer powers of $\widetilde{x}_{\max}:= x_{\max}^{\frac{d-\ell}{d_s}}$ (instead of integer powers of $x_{\max}$).  Let $\widetilde{L}_i$ be the closure of $L_i$ in $\widetilde{\bbC^{4}_w}$.  Near the boundary $\pa \widetilde{L}_i$ of $\widetilde{L}_i$, we can then use the coordinates 
$$
  \widetilde{\xi}_{i}= \xi_i^{\frac{d-\ell}{d_s}}, \quad \omega_{i,i}\in \pa\widetilde{L_i}=\pa \overline{L}_i\cong \bbS^1, \quad (\omega_{i,i_1},\omega_{i,i_2},\omega_{i,i_3})\in \bbC^3,
$$
with $\pa\widetilde{L}_i$ corresponding to 
\begin{equation}
\widetilde{\xi}_{i}=\omega_{i,i_1}=\omega_{i,i_2}=\omega_{i,i_3}=0.
\label{CY.7a}\end{equation}
Then \eqref{CY.5} and Assumption~\ref{CY.13} suggest to consider the weighted blow-up of $\widetilde{L}_i$ with respect to the multiweight $\widetilde{v}_{i}$ which assigns the weight $1$ to the boundary defining function $\widetilde{\xi}_{i}=\xi_i^{\frac{d-\ell}{d_s}}$ and the weight $v_{i,j}$ to $\omega_{i,j}$ for $j\ne i$.   This allows us to consider the space
 \begin{equation}
\widehat{\bbC^{4}_{w}}:= [\cdots[\widetilde{\bbC^{4}_w}; \pa \widetilde{L}_1]_{\widetilde{v}_1} \cdots \pa\widetilde{L}_s]_{\widetilde{v}_s}
\label{CY.7}\end{equation} 
obtained from $\widetilde{\bbC^{4}_w}$ by performing the weighted blow-ups of $\pa\widetilde{L}_1, \ldots, \pa\widetilde{L}_{s}$ in the sense of \cite[\S~4]{CR2024} using respectively the multiweights $\widetilde{v}_1,\ldots,\widetilde{v}_s$.  Since these $p$-submanifolds are disjoint, notice that the order in which we blow up is not important.  Let $\hH_i$ be the boundary hypersurface created by the weighted blow-up of $\pa \widetilde{L}_i$ and let $\hH_{s+1}$ be the boundary hypersurface corresponding to the lift of $\pa\widetilde{\bbC^{4}_w}$ to $\widehat{\bbC^{4}_w}$.  Let $\widehat{C}_{\epsilon}$ be the closure of $C_{\epsilon}$ in $\widehat{\bbC^{4}_w}$.

In terms of the coordinates $\widetilde{\xi}_{i}$ and $\omega_i=(\omega_{i,1},\omega_{i,2},\omega_{i,3},\omega_{i,4})$ near $\pa\widetilde{\bbC^{4}_w}$ on $\widetilde{\bbC^{4}_w}$, the weighted blow-up of $\pa\widetilde{L}_i$ corresponds to introducing the coordinates
\begin{equation}
   \widetilde{\xi}_{i}, \; \omega_{i,i}\in \pa \widetilde{L}_i\cong \bbS^1 \; \quad \mbox{and} \quad \zeta_{i,j}:= \frac{\omega_{i,j}}{\widetilde{\xi}_{i}^{v_{i,j}}}= \xi_i^{w_j-v_{i,j}\frac{d-\ell}{d_s}}z_j \; \mbox{for} \; j\ne i.
\label{CY.8}\end{equation}
More specifically, these are good coordinates near the interior of $\hH_i$.  In these coordinates, equation \eqref{CY.5} takes the form
\begin{equation}
    P_{i,\omega_{i,i}}(\zeta_{i,i_1}, \zeta_{i,i_2}, \zeta_{i,i_3})= \epsilon\left( Q_{\ell}(\omega_i)+ \sum_{q<\ell} \xi_i^{\ell-q}Q_{q}(\omega_i) \right),
\label{CY.9}\end{equation}
so that on the interior of $\hH_i$, it takes the form
\begin{equation}
P_{i,\omega_{i,i}}(\zeta_{i,i_1},\zeta_{i,i_2}, \zeta_{i,i_3})= \epsilon Q_{\ell}(\omega_{i,1},\omega_{i,2},\omega_{i,3},\omega_{i,4})
\label{CY.10}\end{equation}
with $\omega_{i,j}=0$ for $j\ne i$ by \eqref{CY.7a}.
Now, on $\hH_i$, the blow-down map induces a fiber bundle
\begin{equation}
  \hphi_i: \hH_i\to \pa \widetilde{L}_i
\label{CY.11}\end{equation}
and for $\omega_{i,i}\in \pa\widetilde{L}_i$ fixed, $\hphi_i^{-1}(\omega_{i,i})\cap \widehat{C}_{\epsilon}$ is precisely given by \eqref{CY.10}.

\begin{assumption}
For each $i\in\{1,\ldots,s\}$ and each $\omega_{i,i}\in \pa \overline{L}_i$, the equation \eqref{CY.10} yields a smooth affine variety when $\epsilon=\epsilon_0$.  
\label{CY.12}\end{assumption}
\begin{remark}
It is because of this assumption that the singular locus in \eqref{CY.3b} is required to be of complex dimension 1.
\label{CY.12a}\end{remark}

By Assumption~\ref{CY.12}, the closure $\widehat{C}_{\epsilon_0}$ of $C_{\epsilon_0}$ in $\widehat{\bbC^{4}_w}$ is a manifold with corners.  However, as in \cite[\S~5]{CR2024}, it is not necessarily of class $\CI$, but of class $\cC^{k}$ for some nonnegative integer $k\le d_s$.  Nevertheless, by restriction from $\widehat{\bbC^{4}_w}$, there is on $C_{\epsilon}$ a natural ring of `smooth' functions, a ring $\cA_{\phg}(\widehat{C}_{\epsilon_0})\cap L^{\infty}(C_{\epsilon})$ of bounded polyhomogeneous functions, and a ring of $\nQb$-smooth functions $\CI_{\nQb}(C_{\epsilon})$.

We can now state and prove our main result.

\begin{theorem}
Suppose that Assumptions~\ref{CY.3}, \ref{CY.13}, and \ref{CY.12} hold and suppose that 
\begin{equation}
 d-d_s\le \ell <d.
\label{CY.14b}\end{equation}  
If 
$\nu:= \frac{d_s-(d-\ell)}{d_s}$ and $\beta:= \min\{d_s, 4\}$ are such that either $\beta>\frac{2}{1-\nu}$, or else that 
$$
  3<\beta\le \frac{2}{1-\nu}<9,
$$
then the smoothing $C_{\epsilon_0}$ admits a Calabi-Yau warped $\QAC$-metric asymptotic to $g_{C_0}$ with rate $\beta$.  Moreover, near $\hH_i$, this warped $\QAC$-metric is asymptotically modelled on 
\begin{equation}
   d\rho^2+ \rho^2 \hphi_i^*g_{\pa \widetilde{L}_i} + \rho^{2\nu} g_{C_{i,\omega_{i,i},\epsilon_0}},
\label{CY.14a}\end{equation}
where $d\rho^2 + \rho^2 g_{\pa \widetilde{L}_i}$ is a Calabi-Yau cone metric on $L_i$ and $g_{C_{i,\omega_{i,i},\epsilon_0}}$ is a family (as $\omega_{i,i}\in \pa\widetilde{L}_i$ varies) of asymptotically conical Calabi-Yau metrics on the fibers of $\hphi_i|_{\widehat{C}_{\epsilon_0}\cap \hH_i}$ that are seen as $2$-tensors on $\hH_i$ with respect to some choice of connection\footnote{Changing the connection only adds a term of lower order with respect to the model \eqref{CY.14a}.} for the fiber bundle $\hphi_i$.  
\label{CY.14}\end{theorem}
\begin{proof}
Let $r$ be the distance to the origin in $C_0$ with respect to the metric $g_{C_0}$.  Then $r^2$ can be extended to a homogeneous function on all of $\bbC^{4}$.  By Assumption~\ref{CY.3}, near $\pa \widetilde{L}_i$, in the coordinates 
$$ 
  \widetilde{\xi}_{i}= |z_i|^{\frac{d-\ell}{w_i d_s}}, \quad \theta_i:= \arg (z_i), \quad \mbox{and} \quad \varpi_{i,j}:= \frac{z_j}{z_i^{\frac{w_j}{w_i}}} \; \mbox{for} \; j\ne i,
$$
we can assume that the potential $r^2$ takes the form
$$
         r^2= \widetilde{\xi}_{i}^{-\frac{2d_s}{d-\ell}}\left( h_0 +h_2(\theta_i, \varpi_i)+ h_{\ge 3}(\theta_i, \varpi_i) \right), 
$$ 
where $h_0$ is a constant, $h_2$ is smooth in $\theta_i$ and homogeneous of degree $2$ with respect to the $\bbR^+$-action on $\varpi_i=(\varpi_{i,i_1},\varpi_{i,i_2},\varpi_{i,i_3})$ specified by the multiweight $v_i$ as in Assumption~\ref{CY.13} by assigning the weight $v_{i,j}$ to $\varpi_{i,j}$ for $j\ne i$, while $h_{\ge 3}$ is smooth in $\theta_i$ and admits an expansion in homogeneous terms of degree at least $3$ with respect to the $\bbR^+$-action specified by $v_i$.  The fact that $h_0$ is constant is due to the property of $r^2$ being constant along the orbits of the Reeb vector field of $(C_0,g_0)$.  The fact that there is no term $h_1(\theta_i,\varpi_i)$ homogeneous of degree $1$ in the expansion is a consequence of \cite[Lemma~3.1]{CDR2019}.  

  Using \cite[Lemma~5.3]{CR2024}, we can then modify $r^2$ by modifying $h_2(\theta_i,\varpi_i)$ in a compact set for each $i$ (so modifying  $r^2$ in a compact region of $\widehat{\bbC^{4}_w}\setminus \hH_{s+1}$) to obtain a new function $u_{\epsilon_0}$ such that $\frac{\sqrt{-1}}2 \pa \db u_{\epsilon_0}$ is the K\"ahler form of a warped $\QAC$-metric outside some compact set of $C_{\epsilon_{0}}$.  Using again \cite[Lemma~5.3]{CR2024}, we can further modify $u_{\epsilon_0}$ on a compact set to ensure that $\omega_{\epsilon_0}:=\frac{\sqrt{-1}}2 \pa \db u_{\epsilon_0}$ is the K\"ahler form of a warped $\QAC$-metric everywhere on $C_{\epsilon}$.  
On the other hand, there is a natural holomorphic volume form $\Omega^3_{C_{\epsilon}}$ on $C_{\epsilon}$ defined implicitly by 
$$
    dz_1\wedge dz_2\wedge dz_3 \wedge dz_{4}|_{C_{\epsilon}}= \Omega^3_{C_{\epsilon}}\wedge dP|_{C_{\epsilon}}.
$$
The fact that $(C_0,g_{C_0})$ is Calabi-Yau means that the K\"ahler form $\omega_{C_0}$ of $g_{C_0}$ is such that
$$
   \omega_{C_0}^3= c_3\Omega^3_{C_0}\wedge\overline{\Omega}^3_{C_0}
$$ 
for some constant $c_3\in \bbC\setminus \{0\}$.  To obtain the desired Calabi-Yau metric, it suffices then to solve
the equation
$$
(\omega_{\epsilon_0}+\sqrt{-1}\pa\db v)^3= c_3\Omega^3_{C_{\epsilon_0}}\wedge\overline{\Omega}^3_{C_{\epsilon_0}},
$$
that is, to solve the complex Monge-Amp\`ere equation
\begin{equation}
    \log\left( \frac{(\omega_{\epsilon_0}+\sqrt{-1}\pa\db v)^3}{\omega_{\epsilon_0}^3}\right)= -\mathfrak{r}_{\epsilon_0},
\label{CY.17}\end{equation}
where 
$$
     \mathfrak{r}_{\epsilon_0}:= \log \left( \frac{\omega^3_{\epsilon_0}}{c_3\Omega^3_{C_{\epsilon_0}}\wedge \overline{\Omega}^3_{C_{\epsilon_0}}} \right)
$$
is the Ricci potential of $\omega_{\epsilon_0}$.  Proceeding as in the proof of \cite[Lemma~5.10]{CR2024} with \cite[(5.43)]{CR2024} replaced by 
\begin{equation}
\begin{aligned}
d\lrp{P(z)-\epsilon_0Q(z)} &= d(\xi^{-\ell} \lrp{P_{i,\omega_{i,i}}(\zeta_{i,i_1},\zeta_{i,i_2},\zeta_{i,i_3}) -\epsilon_0Q_{\ell}(\omega_i) }) + \cO(\xi^{1-\nu} \xi^{\nu-\ell}) \\
 &= d(\xi^{-\ell} \lrp{P_{i,\omega_{i,i}}(\zeta_{i,i_1},\zeta_{i,i_2},\zeta_{i,i_3}) -\epsilon_0Q_{\ell}(\omega_i) }) + \cO(\widetilde{\xi} \xi^{\nu-\ell}),
\end{aligned}
\label{CY.17b}\end{equation}
where $\cO(\widetilde{\xi} \xi^{\nu-\ell})$ is with respect to the local basis of $1$-forms 
$$
  \frac{d\xi_i}{\xi_i^2}, \quad \frac{d\omega_{i,i}}{\xi_i}, \quad \xi_i^{-\nu}d\zeta_{i,j}, \quad \mbox{for} \quad  j\in\{i_1,i_2,i_3\},
$$ 
we can first check that 
$$
     \mathfrak{r}_{\epsilon_0}\in \widehat{x}_{\max}^{\beta}\CI_{\nQb,1}(C_{\epsilon}).
$$
Using this, we can then proceed as in \cite[\S~6]{CR2024}, in particular \cite[Corollary~6.6]{CR2024}, to obtain the desired result.
\end{proof}
\begin{remark}
As warped $\QAC$-metrics, the metrics of Theorem~\ref{CY.14} are of bounded geometry.  Moreover, by \cite[Remarque~2.11]{Rochon2025}, their curvature is $\mathcal{O}(\widetilde{x}_{\max}^2 \rho^{-2\nu})$ at infinity for $\rho$ (\eg as in the model \eqref{CY.14a}) such that $\rho^{-1}$ is a weighted total boundary defining function on $\widehat{C}_{\epsilon}$  in the sense of \cite[Definition~2.3]{CR2024}.  In particular, the curvature decays like $r^{-2\nu}$ at infinity, though away from the singularities of the tangent cone at infinity, the decay improves to be quadratic.    
\label{bg.1}\end{remark}

\begin{remark}
The hypothesis \eqref{CY.14b} seems to be essential for our method to work.  Indeed, if not, then $\nu<0$ in the local model \eqref{CY.14a}.  Since the metric $g_{C_{i,\omega_i,\epsilon_0}}$ is not flat, this means the model \eqref{CY.14a} does not have bounded curvature.  Now, in an intermediate step of the proof of Theorem~\ref{CY.14}, we need to solve a complex Monge-Amp\`ere equation for a warped $\QAC$-metric having asymptotic model near $\hH_i$ \eqref{CY.14a} with  $g_{C_{i,\omega_i,\epsilon_0}}$ a non-flat asymptotically conical Calabi-Yau metric.  By \cite[\S~2 and 3]{Goldberg-Kobayashi1967}, this implies in particular that  $g_{C_{i,\omega_i,\epsilon_0}}$ has non-vanishing holomorphic bisectional curvature taking both positive and negative values.  This means that the holomorphic bisectional curvature of the corresponding warped $\QAC$-metric is unbounded above and below, seriously compromising the possibility of being able to solve the complex Monge-Amp\`ere equation \eqref{CY.17}.
\label{CY.16}\end{remark}

\section{Examples}\label{ex.0}

In this section, we will give examples of Calabi-Yau cones for which Theorem~\ref{CY.14} applies.  For the first two examples, we will need to invoke the following 
general result indicated to us by Vestislav Apostolov.
\begin{lemma}
A toric Calabi-Yau cone metric on a toric cone with orbifold link is smooth in the sense of orbifolds (away from the apex of the cone).
\label{ap.1}\end{lemma}
\begin{proof}
 Let $Y= \bbR^+\times N$ be a cone of real dimension $2n+2$ with link $(N,D,\hat{T})$ a contact toric orbifold.  By \cite{Lerman2003}, this contact toric orbifold can be completely described in terms of its image $C:=\mu(N)\subset \subset (\bbR^{n+1})^*$ by the moment map $\mu: N\to (\bbR^{n+1})^*$.  It is a polyhedral cone in $(\bbR^{n+1})^*$ of the form
$$
    C= \{ y\in (\bbR^{n+1})^*\; | \; L_i(y)\ge 0\},
$$
where the normals $L_i\in \hat{\mathfrak{t}}=\operatorname{Lie}(\hat{T})\cong \bbR^{n+1}$ above are in the lattice $\Lambda\subset \hat{\mathfrak{t}}$ yielding the identification $\hat{T}= \hat{\mathfrak{t}}/\Lambda$.  These normals are determined by the primitive normals of the orbifold structure of $N$.  

In \cite{Legendre2011}, Legendre describes almost K\"ahler toric cone metrics on $Y$ in terms of a smooth matrix-valued function $\hat{\mathbf{H}}$ defined on $C$ and satisfying the boundary conditions of \cite[Proposition~2.11]{Legendre2011}.  This is formulated in the smooth case, namely the polyhedral cone is assumed to be good in the sense of \cite[Definition~2.4]{Legendre2011}, but the result remains true in the orbifold setting since the conditions of compactification in \cite[Lemma~2.8]{Legendre2011} are stated and proved in \cite[Lemma 2]{ACGT2004} for orbifolds.  In fact, with this observation, the results of \cite{Legendre2011} generalize to orbifolds, in particular a toric cone K\"ahler metric with respect to a Reeb vector field $b\in \hat{\mathfrak{t}}$ on $Y$ is of constant scalar curvature if and only if the matrix-valued function $\hat{\mathbf{H}}$ comes from a solution to Abreu's equation on the labeled transversal polytope $(\Delta_{b},u_b)$ \cite[Corollary~2.17]{Legendre2011}. 

Now, in the Calabi-Yau case, by \cite{MSY2006} (see also \cite[Theorem 3.12]{Legendre2011}), $Y$ admits a unique Reeb direction $b_0\in\hat{\mathfrak{t}}$ with Futaki invariant equal to zero.  For this direction, the transversal polytope $(\Delta_{b_0},u_{b_0})$ is monotone in the sense of \cite[Definition~3.10]{Legendre2011} and its Futaki invariant is zero by \cite[Proposition~3.7]{Legendre2011}.  On the other hand, by \cite[Theorem~1.6]{Legendre2016}, the transversal polytope $(\Delta_{b_0},u_{b_0})$ admits  a unique solution to Abreu's equation which in this case yields a K\"ahler-Einstein metric.  By the discussion above, this induces on $Y$ a K\"ahler Ricci-flat cone metric on $Y$ which is smooth in the sense of orbifolds.
\end{proof}

This allows us to apply Theorem~\ref{CY.14} to a smoothing of a singular irregular Calabi-Yau cone.  More precisely, for integers $k,l \ge 1$, consider the $3$-dimensional hypersurface singularity
$$
C_{k,l} = \{ u^k v^l=wz \} \subset \bbC^4.
$$
This is a toric Gorenstein singularity. Its singular set is given by
$$
\{u=w=z=0\}\cup \{v=w=z=0\},
$$
where the second component is absent if $l=1$ and the first is absent if $k=1$. In particular, when $(k,l)=(2,1)$, one recovers the suspended pinch point singularity (SPP). Note that when $k=l=1$, one recovers the $3$-dimensional $A_1$ singularity (the conifold).

The toric diagram of $C_{k,l}$ is given by the convex hull of the four lattice points
\[
  v_1=(0,0),\:\: v_2=(k,0),\:\: v_3=(l,1), \:\: v_4=(0,1) \in \bbZ^2,
\]
and is identified with the cross section $\mathcal{C} \cap \{x_1=1\}$ of the strongly convex rational polyhedral cone $\mathcal{C} \subset \bbR^3$ generated by $\tilde v_i=(1,v_i)\in\mathbb Z^3$. The dual cone $\mathcal{C}^*$ is generated by the four primitive inward-pointing normals of
$\mathcal{C}$:
\[
  u=(1,0,-1),\:\: v=(0,0,1),\:\: w=(0,1,0),\:\: z=(k,-1,l-k) \in \bbZ^3.
\]
The primitive generators of $\mathcal C^*$ give holomorphic coordinates $(u,v,w,z)$, which satisfy the unique relation
$$
u^k v^l=wz.
$$
Hence the associated affine toric variety is isomorphic to $C_{k,l}=\{u^k v^l=wz\}\subset\bbC^4$.

In the physics literature, Cveti\v{c}-L\"{u}-Page-Pope \cite{CveticLuPagePope2005} and independently  Martelli-Sparks \cite{MartelliSparks2005} constructed a large class of explicit Ricci-flat K\"ahler cone metrics whose links are five-dimensional Sasaki-Einstein orbifolds $L^{p,q,r}$. In particular, $C_{k,l}$ can be identified with the cone over $L^{l,k,l}$. In their notation, $\ell_1,\ell_2,\partial_\phi,\partial_\psi$ are Killing vector fields tangent to the torus fibers; torically, they correspond to the four primitive inward-pointing normals of $\mathcal{C}^*$. They satisfy
$$
p\ell_1+q\ell_2+r\partial_\phi+s\partial_\psi=0,\qquad p+q=r+s.
$$
Up to an affine unimodular transformation of $\mathbb Z^3$, they satisfy
$$
\begin{pmatrix}
  -\ell_1 \\
  -\ell_2 \\
  \partial_\phi \\
  \partial_\psi
\end{pmatrix}
=
\begin{pmatrix}
  1 & 0 & 0 \\
  1 & q & A \\
  1 & r & B\\
  1 & 0 & 1
\end{pmatrix}
\begin{pmatrix}
  e_1 \\
  e_2 \\
  e_3
\end{pmatrix}
$$
with $qB-rA=s$, where $e_{1},\,e_{2},\,e_{3}$ form the standard basis of $\mathbb{R}^{3}$. Setting $(p,q,r)=(l,k,l)$ and $(A,B)=(0,1)$, we recover the toric diagram for $C_{k,l}$. For a similar derivation of the toric diagram of $L^{p,q,r}$, see Butti-Forcella-Zaffaroni \cite{ButtiBorcellaZaffaroni2005}.

Without the explicit use of the Calabi-Yau cone metric itself, the corresponding Reeb vector field can be computed variationally through volume minimization, due to Martelli-Sparks-Yau \cite{MSY2006}, for three-dimensional toric Calabi-Yau cones.

Let $\xi=(\alpha, \beta, \gamma)$ denote the Reeb vector field. For $t \in \bbC^*$, the induced $\bbC^*$ action is given by
\[
  t \cdot (u,v,w,z)
  = (t^{\alpha-\gamma}u,\, t^{\gamma}v,\, t^{\beta}w,\, t^{k\alpha-\beta+(l-k)\gamma}z).
\]
Martelli-Sparks-Yau \cite{MSY2006} showed that the Reeb vector field of a Sasaki-Einstein link is characterized variationally as the unique minimizer of a natural volume functional $Z(\xi)$; see \cite[p.59]{MSY2006} and also \cite[\S 6.2]{GabellaDPhil2011} for the precise definition and its explicit computation in this setting.

The corresponding Reeb vector field is then given by
$$
\xi = (\alpha,\beta,\gamma)
=
\left(
3,\,
\frac{3kl}{2\,(k+l-\sqrt{k^2-kl+l^2})},\,
\frac{3k}{\sqrt{k^2-kl+l^2}+2k-l}
\right).
$$
Equivalently, the induced weights of the coordinates $(u,v,w,z)$ are
\begin{align*}
  wt(u)&=\alpha-\gamma
  =\frac{3(\sqrt{k^2-kl+l^2}+k-l)}{\sqrt{k^2-kl+l^2}+2k-l}, \\
  wt(v) &=\gamma
  =\frac{3k}{\sqrt{k^2-kl+l^2}+2k-l}, \\
  wt(w) &=\beta
  =\frac{3kl}{2\,(k+l-\sqrt{k^2-kl+l^2})}, \\
  wt(z) &= k\alpha-\beta+(l-k)\gamma
  =\frac{3kl}{2\,(k+l-\sqrt{k^2-kl+l^2})}.  
\end{align*}

We now turn to producing Calabi-Yau metrics on certain smoothings of $C_{k,l}$. Perform the following change of variables
$$
u = z_2, \:\: v = z_1, \:\: w = z_3+iz_4, \:\: z = -z_3+iz_4.
$$
Then $C_{k,l}$ is given by the equation
$$
    z_1^lz_2^k+z_3^2+z_4^2=0
$$
in $\bbC^4$.

To satisfy Assumption~\ref{CY.13}, we will consider the following two cases: (1) $l=1, k \ge 2$ and (2) $l=k \ge 2$.  Let us first consider case (1).

\begin{example} 
Suppose that $l=1, k\ge 2$.  First of all, note that the Reeb vector field is irrational, so the Reeb action is irregular. The cone $C_{1,k}$ is singular along the $z_1$-axis, that is, at $z_2=z_3=z_4=0$, so $s=1$. Since $(z_1,z_2,z_3,z_4)$ has multiweight
$$
\left(\frac{2k-1-\sqrt{k^2-k+1}}{k-1},
  \frac{k-2+\sqrt{k^2-k+1}}{k-1},
  \frac{k+1+\sqrt{k^2-k+1}}{2},
  \frac{k+1+\sqrt{k^2-k+1}}{2}
\right),
$$
the defining polynomial $P(z)$ has weighted degree $d=k+1+\sqrt{k^2-k+1}$.

The polynomial $P_{1,\omega_1}$ takes the form
$$
P_{1,\omega_{1,1}}(\omega_{1,2},\omega_{1,3},\omega_{1,4}):= \omega_{1,1}\omega_{1,2}^k+ \omega_{1,3}^2+\omega_{1,4}^2.
$$
The equation $P_{1,\omega_{1,1}}(\omega_{1,2},\omega_{1,3},\omega_{1,4})=0$ has an $A_{k-1}$-singularity at the origin.  As such, it is a Calabi-Yau cone with $\bbR^+$-action induced by the multiweight $(2,k,k)$ for $(\omega_{1,2},\omega_{1,3},\omega_{1,4})$.  Hence $d_s=2k$ and Assumption~\ref{CY.13} holds.  By Lemma~\ref{ap.1}, the Calabi-Yau cone metric $g_{C_{k,l}}$ on $C_{k,l}$ is smooth in the sense of orbifolds on $C_{k,l}\setminus \{0\}$, confirming that Assumption~\ref{CY.3} holds.
  
Hence, we can consider a smoothing of weighted degree
$$
   d= k+1+\sqrt{k^2-k+1} > \ell\ge d-d_s= -k+1+\sqrt{k^2-k+1}.
$$
For instance, consider the smoothing of $C_{1,k}$ given by perturbing by the polynomial $Q=z_1^m+1$ with $1 \le m \le 2k-1$. Then the weighted degree $\ell$ of $Q$ satisfies the above inequality. Thus the smoothing is defined by
\begin{equation}
   z_1z_2^k+z_3^2+z_4^2=z_1^m+1.
\label{sm.1}\end{equation}
In this case, equation \eqref{CY.10} with $i=1$ and $\epsilon=1$ corresponds to the smoothing
$$
       \omega_{1,1} \zeta_{1,2}^k+ \zeta_{1,3}^2+ \zeta_{1,4}^2= \omega_{1,1}^m.
$$ 
Hence, Assumption~\ref{CY.12} holds with $\epsilon_0=1$. However, to apply Theorem~\ref{CY.14} with $\beta=\min\{d_s,4\} = 4$ and $\nu=\frac{d_s-(d-\ell)}{d_s}$, one needs to restrict $m$ to a smaller range. For small $k$, one has the following range for $m$:
\begin{itemize}
\item $k=2$: $m=1,2,3$,
\item $k=3$: $m=1,2,3,4$,
\item $k=4$: $m=1,2,3,4,5,6$.
\end{itemize}

By Theorem~\ref{CY.14}, for each admissible $m$ we obtain a Calabi-Yau warped $\QAC$-metric on the smoothing given by \eqref{sm.1}.
\label{SPP.1}\end{example}

Next we consider case (2). 
\begin{example}
If we suppose that $k=l \ge 2$, then the multiweight for $(z_1,z_2,z_3,z_4)$ simplifies to
$$
\left(\frac{3}{2},\frac{3}{2},\frac{3k}{2},\frac{3k}{2}\right),
$$
and the defining polynomial
$$
P(z) = z_1^kz_2^k+z_3^2+z_4^2
$$
has weighted degree $d=3k$ and the cone $C_{k,k}$ is singular along the $z_1$-axis and the $z_2$-axis, so $s=2$.  For $i\in\{1,2\}$, the polynomial $P_{i,\omega_{i,i}}(\omega_i)$ takes the form
$$
         P_{1,\omega_{1,1}}(\omega_{1,2},\omega_{1,3},\omega_{1,4})= \omega_{1,1}^k\omega_{1,2}^k+\omega^2_{1,3}  + \omega_{1,4}^2\quad  \mbox{and} \quad   P_{2,\omega_{2,2}}(\omega_{2,1},\omega_{2,3},\omega_{2,4})= \omega_{2,1}^k\omega_{2,2}^k+\omega^2_{2,3}  + \omega_{2,4}^2.
$$
As in the previous example, the corresponding cones $C_{1,\omega_{1,1}}$ and ${C_{2,\omega_{2,2}}}$ are Calabi-Yau with an $A_{k-1}$ singularity at the origin.  Thanks to Lemma~\ref{ap.1}, this means that Assumption~\ref{CY.3} holds, while Assumption~\ref{CY.13} holds with $d_s=2k$.  
Correspondingly, we have $\beta = 4$ and $\nu = \frac{\ell-k}{2k}$. The condition $d > \ell > d-d_s$ implies that $k\le\ell < 3k$. The condition that either $\beta > \frac{2}{1-\nu}$ or $\beta \le \frac{2}{1-\nu} < 9$ then forces $\ell$ to satisfy
$$
k \le \ell <\frac{23k}{9}.
$$

For instance, we could take $Q(z) = z_1^m+z_2^m+c$ for $\frac{2k}3 \le m<\frac{46 k}{27}$ and $c\in \bbC$.  For $c$ generic, 
$$
     P(z)=Q(z)
$$  
will then define a smooth hypersurface.
In this case, equation \eqref{CY.10} with $i=1$ and $\epsilon=1$ takes the form
$$
      \omega_{1,1}^k\zeta_{1,2}^k+ \zeta_{1,3}^2+ \zeta_{1,4}^2= \omega_{1,1}^m,
$$
since $\xi_1=\omega_{1,2}=0$ on $\hH_1$, while for $i=2$ and $\epsilon=1$, it takes the form
$$
  \omega_{2,2}^k\zeta_{2,1}^k+ \zeta_{2,3}^2+ \zeta_{2,4}^2= \omega_{2,2}^m,
$$
since $\xi_2=\omega_{2,1}=0$ on $\hH_2$, showing that Assumption~\ref{CY.12} holds in this case.   We can therefore apply Theorem~\ref{CY.14} to obtain Calabi-Yau warped $QAC$-metrics on this smoothing.
\label{SPP.2}\end{example}

\begin{remark}
The above two examples are significant in that neither follows from existing constructions of complete Calabi--Yau metrics; in particular, the irregular tangent cone at infinity is neither a product nor an orbifold.  They also differ from the examples of \cite{Nghiem}, since those are of complex dimension strictly larger than $3$. 
\end{remark}
\begin{remark}
By Remark~\ref{CY.16}, our method does not apply to
the Milnor fiber
$$
z_1z_2^2+z_3^2+z_4^2+1=0.
$$
It remains an intriguing open question whether a complete Calabi-Yau metric exists on this manifold.
\end{remark}

We can also consider smoothings of orbifold Calabi-Yau cones.

\begin{example}
For an integer $n\ge 2$, consider the Calabi-Yau cone $\bbC^3/\bbZ_n^2$ (\cf \cite[Example 9.9.9]{Joyce} when $n=2$) with the $\bbZ_n\oplus\bbZ_n$ action given by
$$
      (e^{\frac{2\pi k_1 \sqrt{-1}}n},e^{\frac{2\pi k_2 \sqrt{-1}}n})\cdot (x_1,x_2,x_3)=(e^{\frac{2\pi k_1 \sqrt{-1}}n}x_1, e^{\frac{2\pi (k_2-k_1) \sqrt{-1}}n}x_2, e^{-\frac{2\pi k_2 \sqrt{-1}}n}x_3).
$$
Setting $z_i:=x_i^n$ for $i\in\{1,2,3\}$ and $z_4:= x_1x_2x_3$ allows us to describe this cone as the affine hypersurface
$$
 z_1z_2z_3-z_4^n=0
$$
in $\bbC^4$.  In this case, the multiweight is given by $w=(n,n,n,3)$ and the weighted degree of the polynomial is $d=3n$.  The cone is singular along the $z_1$, $z_2$, and $z_3$ axes.  The Calabi-Yau cone metric is just the metric induced by the Euclidean metric on $\bbC^3$, so it is automatically smooth in the sense of orbifolds, showing that Assumption~\ref{CY.3} holds with $s=3$.  For $i\le 3$, the polynomial $P_{i,\omega_{i,i}}$ takes the form
$$
     P_{i,\omega_{i,i}}(\omega_{i,i_1},\omega_{i,i_2},\omega_{i,i_3})=\omega_{i,1}\omega_{i,2}\omega_{i,3}-\omega_{i,4}^n.
$$
Moreover, the $\bbR^+$-action on $C_{i,\omega_{i,i}}$ is induced by a weighted action with weight $2$ for $\omega_{i,4}$ and weight $n$ for the other two variables. In particular, $P_{i,\omega_{i,i}}$ is homogeneous of weighted degree $d_s=2n$ for each $i\le 3$, ensuring that Assumption~\ref{CY.13} holds with $d_s=2n$.  

We can therefore consider the smoothing obtained by perturbing by a polynomial of weighted degree 
$$
\ell\ge d-d_s=3n-2n=n.
$$  
For $\ell= 2n$, we can in particular consider the perturbation $Q=z_1^2+z_2^2+z_3^2+1$ of weighted degree $2n$, namely the smoothing given by the equation
\begin{equation}
z_1z_2z_3-z_4^n=z_1^2+z_2^2+z_3^2+1.
\label{sm.2}\end{equation}
In this case, for $i=1$, equation \eqref{CY.10} takes the form
\begin{equation}
\omega_{1,1}\zeta_{1,2}\zeta_{1,3}-\zeta_{1,4}^n= \omega_{1,1}^2,
\label{an.1}\end{equation}
since $\xi_1=\omega_{1,2}=\omega_{1,3}=0$ on $\hH_1$. In particular, it is smooth.  There are similar results for the equation \eqref{CY.10} when $i\in\{2,3\}$, so that Assumption~\ref{CY.12} holds.  Hence, Theorem~\ref{CY.14} can be applied to the smoothing \eqref{sm.2} with $\beta= \min\{2n, 4\}=4$ and $\nu=\frac12$.  
\end{example}

\begin{example}
In the previous example, we can also take $\ell=n$ with the smoothing 
$$
z_1z_2z_3-z_4^n=z_1+z_2+z_3.
$$  
One can check that Assumption~\ref{CY.12} holds again for this smoothing.  In this case, $\nu=\frac{d_s-(d-\ell)}{d_s}=0$, which means that Theorem~\ref{CY.14} with $\beta=\min\{2n,4\}$ produces a Calabi-Yau $\QAC$-metric (i.e. there is no warping factor) on the smoothing.  
\end{example}

\section{Complete intersections} \label{ci.0}

In this section, we will illustrate how our method can also be applied in a case where the Calabi-Yau cone is given by a complete intersection.  Consider the Calabi-Yau cone $\bbC^3/(\bbZ_2\oplus \bbZ_4)$ with action of $\bbZ_2\oplus\bbZ_4$ given by
$$
   ((-1)^{k_1},\sqrt{-1}^{k_2})\cdot (x_1,x_2,x_3)= ((-1)^{k_1}x_1, (-1)^{k_1}\sqrt{-1}^{k_2}x_2,\sqrt{-1}^{-k_2}x_3).
$$
Setting $z_1=x_1^2$, $z_2=x_2^4$, $z_3=x_3^4$, $z_4=x_1x_2x_3$, and $z_5=x_2^2x_3^2$ yields a description of the cone as the complete intersection
$$
    P_1(z):=z_1z_5-z_4^2=0 \quad \mbox{and} \quad P_2(z):=z_2z_3-z_5^2=0
$$
in $\bbC^5$.  The natural $\bbR^+$-action on $\bbC^5$ is specified by the multiweight $w=(2,4,4,3,4)$, so $P_1$ and $P_2$  have respectively weighted degrees $d_1=6$ and $d_2=8$.

In this case, the singularities are given by the $z_1$-axis $L_1$, the $z_2$-axis $L_2$, and the $z_3$-axis $L_3$, where
$$
  L_i= \{ (z_1,\ldots,z_5)\in \bbC^5\; | \; z_j=0 \; \mbox{for} \; j\ne i\}.
$$
Using the coordinates
\begin{equation}
    \xi_i:= \frac{1}{|z_i|^{\frac{1}{w_i}}}\quad \mbox{and}  \quad \omega_i:=(\omega_{i,1},\omega_{i,2},\omega_{i,3},\omega_{i,4},\omega_{i,5})\quad \mbox{with} \quad \omega_{i,j}:= \frac{z_j}{|z_i|^{\frac{w_j}{w_i}}},
\label{cico.1}\end{equation}
the singularities along the $z_1$-axis are described by the equations
\begin{equation}
  P_{1,1,\omega_{1,1}}(\omega_{1,2},\omega_{1,3},\omega_{1,4},\omega_{1,5}):= \omega_{1,1}\omega_{1,5}-\omega_{1,4}^2=0 \quad \mbox{and} \quad P_{2,1,\omega_{1,1}}(\omega_{1,2},\omega_{1,3},\omega_{1,4},\omega_{1,5}):= \omega_{1,2}\omega_{1,3}-\omega_{1,5}^2=0.  
\label{ci.1}\end{equation}
In particular, the second equation 
\begin{equation}
   \omega_{1,2}\omega_{1,3}-\omega_{1,5}^2=0
\label{ci.1b}\end{equation}
is an affine hypersurface in $\bbC^3$ admitting a Calabi-Yau cone metric with $\bbR^+$-action induced by the $\bbR^+$-action on $\bbC^3$ assigning the weight $2$ to $\omega_{1,2},\omega_{1,3}$ and $\omega_{1,5}$.  In particular, the weighted degree of $P_{1,1,\omega_{1,1}}$ with respect to this weighted action is $d_s=4$.  

The singularities along the $z_2$-axis are instead described by the equations
\begin{equation}
P_{1,2,\omega_{2,2}}(\omega_{2,1},\omega_{2,3},\omega_{2,4},\omega_{2,5}):= \omega_{2,1}\omega_{2,5}-\omega_{2,4}^2=0 \quad \mbox{and} \quad 
P_{2,2,\omega_{2,2}}(\omega_{2,1},\omega_{2,3},\omega_{2,4},\omega_{2,5}):= \omega_{2,2}\omega_{2,3}-\omega_{2,5}^2=0.
\label{ci.2}\end{equation}
The first equation
\begin{equation}
\omega_{2,1}\omega_{2,5}-\omega_{2,4}^2=0
\label{ci.2b}\end{equation}
describes an affine hypersurface in $\bbC^3$ admitting a Calabi-Yau cone metric with $\bbR^+$-action induced by the $\bbR^+$-action on $\bbC^3$ assigning the weight $2$ to each variable.   In particular, the weighted degree of $P_{1,2,\omega_{2,2}}$ with respect to this weighted action is again $d_s=4$. Finally, the singularities along the $z_3$-axis admit a similar description with $z_2$ and $z_3$ interchanged.

Consider the smoothing $C_1$ of $C_0$ given by
\begin{equation}
   P_1(z)=Q_1(z) \quad \mbox{and} \quad P_2(z)=Q_2(z),
\label{ci.2c}\end{equation}
with $Q_1(z):=z_2+z_3$ and $Q_2(z):=z_1^3+1$ polynomials respectively of weighted degree $\ell_1=4$ and $\ell_2=6$. In the coordinates \eqref{cico.1} near $\pa\overline{\bbC^5_w}$, these equations take the form
\begin{equation}
    P_1(\omega_i)= \xi_i^2 Q_{1}(\omega_i) \quad \mbox{and} \quad P_2(\omega_i)= \xi_i^2 \omega_{i,1}^3+ \xi_i^8 
\label{ci.4}\end{equation}
and are singular along $\pa\overline{L}_1,\pa\overline{L}_2$, and $\pa \overline{L}_3$. To resolve these singularities, we can consider the blow-ups of $\pa (\overline{L_1+L_4})$ and $\pa(\overline{L_2+L_3})$ in $\overline{\bbC^5_w}$, but to ease the comparison with the resolution used in \S~\ref{mr.0}, we will first introduce the manifold with boundary $\widetilde{\bbC^5_w}$, which as a topological space is homeomorphic to $\overline{\bbC^5_w}$, but with space of smooth functions corresponding to smooth functions on $\bbC^5$  admitting a smooth expansion in integer powers of $\widetilde{x}_{\max}:=x_{\max}^{\frac{2}{d_s}}=\xi^{\frac12}$ (instead of integer powers of $x_{\max}$) for $x_{\max}$ a choice of boundary defining function for $\overline{\bbC^5_w}$.  Let $\widetilde{L_1+L_4}$ and $\widetilde{L_2+L_3}$ be the closure of $L_1+L_4$ and $L_2+L_3$ in $\widetilde{\bbC^5_w}$.   We shall then consider the manifold with corners
\begin{equation}
  \widehat{\bbC^5_w}:= [[\widetilde{\bbC^5_w}; \pa(\widetilde{L_1+L_4})]_{\widetilde{v}_1} \pa(\widetilde{L_2+L_3})]_{\widetilde{v}_2}
\label{ci.3}\end{equation}  
obtained by blowing up $\pa(\widetilde{L_1+L_4})$ and $\pa(\widetilde{L_2+L_3})$ in the sense of \cite[\S~4]{CR2024} with respect to the multiweight $\widetilde{v}_1$ and $\widetilde{v}_2$, 
where $\widetilde{v}_1$ is the multiweight assigning the weight $1$ to $\widetilde{\xi}_i=\xi_i^{\frac12}$ and the weight $2$ to the  variables $\omega_{i,2},\omega_{i,3}$, and $\omega_{i,5}$ in the coordinates \eqref{cico.1} for $i\in\{1,4\}$, while $\widetilde{v}_2$ is the multiweight assigning the weight $1$ to $\widetilde{\xi}_{i}:=\xi_i^{\frac12}$ and the weight $2$ to the variables $\omega_{i,1},\omega_{i,4}$, and $\omega_{i,5}$ in the coordinates \eqref{cico.1} for $i\in\{2,3\}$.
Since $\pa(\widetilde{L_1+L_4})$ and $\pa(\widetilde{L_2+L_3})$ are disjoint, notice that their weighted blow-ups commute.  Let $\hH_1$ and $\hH_2$ be the boundary hypersurfaces created by the blow-ups of $\pa(\widetilde{L_1+L_4})$ and $\pa(\widetilde{L_2+L_3})$.  Let $\hH_3$ be the lift of $\pa\widetilde{\bbC^5_w}$ to $\widehat{\bbC^5_w}$.  In terms of the coordinates $(\xi_i,\omega_i)$ for $i=1$, the blow-up of $\pa(\widetilde{L_1+L_4})$ corresponds to introducing the coordinates
$$
       \widetilde{\xi}_1=\xi_1^{\frac12}, \quad\omega_{1,1}, \quad\zeta_{1,2}:= \frac{\omega_{1,2}}{\xi_1}, \quad \zeta_{1,3}:= \frac{\omega_{1,3}}{\xi_1}, \quad \omega_{1,4}, \quad \zeta_{1,5}:=\frac{\omega_{1,5}}{\xi_1}.
$$
These are good coordinates on the interior of $\hH_1$.  In terms of these coordinates, the equations \eqref{ci.4} take the form
\begin{equation}
\widetilde{\xi}_1^2\omega_{1,1}\zeta_{1,5}-\omega_{1,4}^2= \widetilde{\xi}_1^4Q_1(\omega_1) \quad \mbox{and} \quad  \zeta_{1,2}\zeta_{1,3}-\zeta_{1,5}^2= \omega_{1,1}^3+ \widetilde{\xi}_1^{12}.
\label{ci.5}\end{equation}
On $\hH_1$, these equations restrict to give
\begin{equation}
   \omega_{1,4}=0 \quad \mbox{and} \quad  \zeta_{1,2}\zeta_{1,3}-\zeta_{1,5}^2=\omega_{1,1}^3,
\label{ci.6}\end{equation}
which is clearly a smooth hypersurface for each value of $\omega_{1,1}\in \pa\widetilde{L}_1$.  

On the other hand, in terms of the coordinates \eqref{cico.1} for $i\in\{2,3\}$, the blow-up of $\pa(\widetilde{L_2+L_3})$ amounts to introducing the coordinates
$$
    \widetilde{\xi}_{i}=\xi_i^{\frac12},\quad \zeta_{i,1}:= \frac{\omega_{i,1}}{\xi_i}, \quad\omega_{i,2},\quad \omega_{i,3}, \quad \zeta_{i,4}:=\frac{\omega_{i,4}}{\xi_i}, \quad \zeta_{i,5}:= \frac{\omega_{i,5}}{\xi_i}.
$$
These coordinates are good coordinates near the interior of $\hH_2$.  In terms of these coordinates, the equations \eqref{ci.4} take the form
\begin{equation}
\zeta_{i,1}\zeta_{i,5}-\zeta_{i,4}^2= \omega_{i,2}+\omega_{i,3} \quad \mbox{and} \quad \omega_{i,2}\omega_{i,3}- \widetilde{\xi}_{i}^4\zeta_{i,5}^2= \widetilde{\xi}_{i}^{10}\zeta_{i,1}^3 + \widetilde{\xi}_{i}^{16}.
\label{ci.7}\end{equation}
When we take $i=2$ and those equations are restricted to $\hH_2$, this yields the family of smooth affine hypersurfaces
\begin{equation}
\zeta_{2,1}\zeta_{2,5}-\zeta_{2,4}^2= \omega_{2,2}+\omega_{2,3} \quad \mbox{with} \quad \omega_{2,2}\in \pa \widetilde{L}_2\cong \bbS^1, \; \omega_{2,3}=0.
\label{ci.8}\end{equation}
For $i=3$, the restriction to $\hH_2$ is instead given by
\begin{equation}
\zeta_{3,1}\zeta_{3,5}-\zeta_{3,4}^2= \omega_{3,2}+\omega_{3,3} \quad \mbox{with} \quad \omega_{3,2}=0, \; \omega_{3,3}\in \pa\widetilde{L}_3\cong \bbS^1.
\label{ci.9}\end{equation}
If $\widehat{C}_1$ denotes the closure of the smoothing $C_1$ in $\widehat{\bbC^5_w}$, this shows that $\widehat{C}_1$ is a manifold with corners with boundary hypersurfaces given by $\widehat{C}_1\cap \hH_1$, $\widehat{C}_1\cap \hH_2$ (having two connected components at $\omega_{3,2}=0$ and $\omega_{2,3}=0$), and $\widehat{C}_1\cap\hH_3$.  In this case, since $x_{\max}= \widetilde{x}_{\max}^2$ is an integer power of $\widetilde{x}_{\max}$, notice that $\widehat{C}_1$ is automatically a smooth $p$-submanifold of $\widehat{\bbC^5_w}$.   

In this complete intersection setting, we see then that the natural generalization of Assumption~\ref{CY.3} still holds, while Assumption~\ref{CY.12} corresponds to \eqref{ci.7}, \eqref{ci.8}, and \eqref{ci.9} defining smooth affine varieties.  For Assumption~\ref{CY.13}, the correct adaptation is this assumption applied to the affine hypersurfaces \eqref{ci.1b} and \eqref{ci.2b} with weighted degree $d_s=4$. This yields the following result.
\begin{theorem}
The smoothing $C_1$ in \eqref{ci.2c} admits a Calabi-Yau warped $\QAC$-metric with compactification $\widehat{C}_1$ and with weight 
$\nu= \frac12$.
\label{ci.10}\end{theorem}
\begin{proof}
We can proceed essentially as in the proof of Theorem~\ref{CY.14} with $\beta=4$ and $\nu=\frac12$.
\end{proof}

\bibliography{QAC_smoothing}
\bibliographystyle{amsplain}

\end{document}